\documentclass[letterpaper, 10 pt, conference]{ieeeconf}
\IEEEoverridecommandlockouts                            % This command is only
\usepackage{url}
\usepackage{verbatim}
\usepackage{graphics} % for pdf, bitmapped graphics files
\usepackage{mathptmx} % assumes new font selection scheme installed
\usepackage{times} % assumes new font selection scheme installed
\usepackage{amsmath} % assumes amsmath package installed
\usepackage{amssymb}  % assumes amsmath package installed
\usepackage{amsfonts,amsmath}
\usepackage{amssymb}
\usepackage{acronym}
\usepackage[dvips]{graphicx}
\usepackage{graphicx}
\usepackage{epsfig}
\usepackage{cite}
\usepackage{algorithm}
\usepackage[dvipsnames,usenames]{color}
\definecolor{mypink1}{rgb}{0.858, 0.188, 0.478}

\def\fskip#1{}

\newtheorem{theorem}{Theorem}

\newtheorem{assumption}{Assumption}

\newtheorem{lemma}{Lemma}

\newtheorem{proposition}[theorem]{Proposition}

\def\argmin{\mathop{\rm argmin}}
\def\argmax{\mathop{\rm argmax}}

\renewcommand{\theenumi}{\arabic{enumi} )}
\newcommand{\Real}{\ensuremath{\mathbb{R}}}

\newcommand{\thetahat}{\widehat{\theta}}

\def\be{\begin{enumerate}}
\def\ee{\end{enumerate}}

\def\st{\mbox{subject to}}

		\def\bkE{{\rm I\kern-.17em E}}
		\def\bk1{{\rm 1\kern-.17em l}}
		\def\bkD{{\rm I\kern-.17em D}}
		\def\bkR{{\rm I\kern-.17em R}}
		\def\bkP{{\rm I\kern-.17em P}}
		\def\bkY{{\bf \kern-.17em Y}}
		\def\bkZ{{\bf \kern-.17em Z}}

\def\k{\kappa}

\def\xbar{\bar{x}}

\def\R{\mathbb{R}}

\def\argmin{\mathop{\rm argmin}}

\title{\LARGE \bf On the rate analysis of inexact augmented Lagrangian schemes for convex optimization problems with misspecified constraints}
\author{H.~Ahmadi \and N.~S.~Aybat \and  U.~V.~Shanbhag\thanks{Ahmadi,
	Aybat, and Shanbhag are
with the Department of Industrial and Manufacturing
Engineering, respectively at the Pennsylvania State University,
	University Park, PA-16803. They are reachable at ({\tt
		ahmadi.hesam@gmail.com,nsa10,udaybag@psu.edu}) and their research has been
partially funded by NSF Grant CMMI-1400217.}}
\begin{document}
\maketitle
\thispagestyle{empty}
\pagestyle{empty}
\begin{abstract}
We consider a misspecified optimization problem that requires minimizing
of a convex function $f(x;\theta^*)$ in $x$  over a constraint set represented by
$h(x;\theta^*) \leq 0$ where $\theta^*$ is an unknown (or misspecified) vector of
parameters. Suppose $\theta^*$ is learnt by a distinct process that generates a
sequence of estimators $\theta_k$, each of which is an increasingly accurate
approximation of $\theta^*$. We develop a
%constant joint
first-order augmented Lagrangian scheme for computing an optimal solution $x^*$ while simultaneously learning $\theta^*$.
\end{abstract}
\section{Introduction}
Consider an
optimization problem in $n-$dimensional space %given by (${\cal C}(\theta^*)$)
defined as follows:
\begin{align}
\mathcal{X}^*(\theta^*):=\argmin_{x \in X \cap \mathcal{H}(\theta^*)} f(x;\theta^*), \tag{${\cal C}(\theta^*)$}
\end{align}
where $\theta^*\in\R^d$ denotes the parametrization of the objective and
constraints. While traditionally, optimization research has considered settings where $\theta^*$ is available a priori, two related
	problems of interest have considered regimes where either the
		parameter is unavailable ({\it robust} optimization) or when it
		is uncertain ({\it stochastic} optimization:)\\
%	stochastic optimization has considered regimes where the parameter
%is uncertain} there also have been concerted efforts
%to consider least two distinct generalizations when
%for the case ${\cal H}(\theta^*)= \Real^n$ and \ser{$f$ depends on the \emph{unknown} parameter}, namely {\it robust} and {\it stochastic} optimization.\\
\noindent {\bf Robust approaches~\cite{bental09robust}.} For instance, when
$\theta^*$ is unavailable, but one has access to an associated uncertainty set
${\cal T}$, %corresponding to \ser{$\theta^*$},
then in robust optimization, %approaches
the worst-case value of the objective is minimized:
\begin{align}\min_{x \in X } \ \max_{ \theta \in {\cal T}} f(x;\theta).
	\tag{Robust Optimization}
\end{align}
\noindent {\bf Stochastic approaches~\cite{shapiro09lectures}.} An alternative approach considers
an uncertain regime where $\theta: \Omega \to \Real^d$ is an
$d-$dimensional random vector %variable
defined on a suitable probability space.
The resulting stochastic optimization schemes consider the minimization
of an expectation:
\begin{align}\min_{x \in X } \ \mathbb{E}[f(x;\theta)].
	\tag{Stochastic Optimization}
\end{align}
In this paper, we consider a different approach in which the parameter vector $\theta$ has a
nominal or true value $\theta^*$ obtainable by solving a
suitably defined learning problem:
\begin{align}
\min_{\theta \in \Theta} \  \ell(\theta). \tag{${\cal E}$}
\end{align}
Instances of such problems routinely arise when $\theta^*$ is
idiosyncratic to the problem and may be learnt by the aggregation of
data; examples include the following: the learning of covariance
matrices associated with a collection of stocks, efficiency parameters
associated with machines on a supply line, and demand parameters
associated with a supply chain. A natural approach in this case is to
first estimate $\theta^*$ with high accuracy and then to solve the
parametrized problem. Yet, in many
instances, this {\it sequential} approach cannot be adopted for several reasons: (i) observations unavailable a priori and appear in a
streaming fashion;
(ii) the learning problem can be large, precluding a highly accurate a priori
resolution;
(iii) unless the learning problem can be solved {\it exactly} in finite
time, any sequential scheme may provide approximate solutions, at best.
\\
Accordingly, we consider the development of schemes that generate
sequences $\{x_k\}, \{\theta_k\}$ such that
\begin{align*}
%\pmat{\|\theta_k - \theta^*\| \\
%			\|x_k -x^*\|} \to 0  \mbox{ as } k \to
%			\infty,
\|\theta_k - \theta^*\| \to 0,\quad d_{\mathcal{X}^*(\theta^*)(x_k) \to 0\ \mbox{ as } k \to \infty,}
\end{align*}
where $\theta^*$ is the unique solution of (${\cal E}$) and for a given closed convex set $\cal X$, $d_{\mathcal{X}}(x)\triangleq\min_{s\in \mathcal{X}} \|x-s\|$
%$x^*$ is a solution to (${\cal C}(\theta^*)$).
denotes the distance function to $\cal X$. {This framework originates from prior work
%that has
on stochastic optimization/variational inequality
	problems~\cite{jiang13solution} and stochastic Nash
		games~\cite{jiang16nash}. In recent work, %as well as
%a refinement of
the rate statements derived in~\cite{jiang13solution} are refined to the
deterministic regime~\cite{ahmadi14data}.}
In~\cite{ahmadi15misspecified}, misspecification in the constraints is addressed in a convex
regime via variational inequality approaches; in sharp contrast, in this paper,
we develop a misspecified analog of the augmented Lagrangian scheme for
misspecified convex problems in which both the objective and the constraints are misspecified.
Augmented Lagrangian schemes are rooted in the seminal papers by Hestenes~\cite{hestenes1969multiplier}
and Powell~\cite{Powell69_1J}, and their relation to proximal-point
methods was established by Rockafellar~\cite{Rockafellar73_1J,rockafellar1973dual}. Recently, there has been a renewed examination
of such techniques in convex regimes, with an emphasis on deriving convergence rates~\cite{aybat2013augmented,lan15_1J,1506.05320 }.

In this paper, we develop an analog of the traditional augmented
Lagrangian scheme in which the subproblems are solved with increasing
exactness. Our contributions include  rate statements for the dual
suboptimality, the primal infeasibility, and the primal suboptimality in
this misspecified regime.  Throughout, our focus will be on the problem
(${\cal C}(\theta^*)$) when $ {\cal H}(\theta^*) \triangleq \{ x:
h(x;\theta^*) \leq 0 \}$, i.e.,
%and
%\sout{$f(x;\theta^*) = p(x;\theta^*) + q(x;\theta^*), $ as given by}
\begin{align}
\tag{${\cal C}(\theta)$}
\begin{aligned}
\quad \min_{x}\quad f(x;\theta)&\\
 \mbox{\st}\quad h(x;\theta)&\leq 0,\   x \in X,
	\end{aligned}
\end{align}
where $f:\R^n\times \Theta \to \R\cup\{+\infty\}$,
%\sout{$q:{\cal D}_q\times\Theta\to \R$},
$h:\R^n\times \Theta\to \R^m$ and $\theta \in \Theta \subseteq \R^d$ denotes an estimate for the misspecified parameter $\theta^*$. %where $\theta \in \Theta \subseteq \R^d$,
Throughout, we assume that %$X\subset\R^n$ is a compact convex set %\sout{and ${\cal D}_q$ are subset of $\R^n$.}
%\sout{for purposes of well-posedness}
${\cal C}(\theta^*)$ has
a finite optimal value, given by $f^*$, the corresponding Lagrangian
dual problem has a solution, denoted by $\lambda^*$, and there is {\it
	no} duality gap. The remainder of the paper comprises of three
	sections. We provide preliminaries in
Section~\ref{Sec:II}, the main rate statements in Section~\ref{Sec:III},
and conclude in Section~\ref{Sec:4}.
%\ser{COMMENT: What are our assumptions on $\Theta$? Compact?}
\section{Preliminaries}\label{Sec:II}
The problem ${\cal C}(\theta)$ is equivalent to
\begin{align}
\min & \left\{f(x;\theta):%+{\rho\over 2}\|h(x;\theta))+z\|^2:
h(x;\theta)+z=0, \quad x\in X,\quad z\in
\mathbb{\R}_+^m\right\}.\label{formul:1}
\end{align}
Let $\lambda\in \R^m$ denote the vector dual variables corresponding to the equality constraints in \eqref{formul:1}. For any given $\rho>0$, define the augmented Lagrangian function for \eqref{formul:1}, ${\cal L}_{\rho}(x,\lambda;\theta)$, such that $\mbox{dom}\,{\cal L}_\rho=X\times\R^m\times\Theta$ and
{
$$ {\cal L}_{\rho}(x,\lambda;\theta)\triangleq \min_{z\in \R_+^m}
\left[f(x;\theta)+\lambda^\top (h(x;\theta)+z)+\tfrac{\rho}{2}\|h(x;\theta)+z\|^2\right].$$
}%
Through a rearrangement of terms, it can be shown that
\begin{align}
{\cal L}_\rho(x,\lambda;\theta) & =f(x;\theta)+\tfrac{\rho}{2}\min_{z\in \R_+^m} \Big\|h(x;\theta)+z+{\lambda \over
		\rho}\Big\|^2-{\|\lambda\|^2\over 2\rho}\notag\\
&=f(x;\theta)+\tfrac{\rho}{2} d^2_{\mathbb{R}^m_-}\left(\tfrac{\lambda}{\rho}+h(x;\theta)\right) -{\|\lambda\|^2\over 2\rho},\label{Aug_L_formu}	
\end{align}
where $d_{\mathcal{X}}(x)\triangleq\min_{s\in \mathcal{X}} \|x-s\|$, and $d^2_{\mathcal{X}}(x)\triangleq\left(d_{\mathcal{X}}(x)\right)^2$.
 For $\rho=0$, let ${\cal L}_0(x,\lambda;\theta)$ denote the Lagrangian function: %defined as
\begin{align*}
 {\cal L}_0(x,\lambda;\theta)\triangleq \begin{cases}
 f(x;\theta)+\lambda^\top h(x;\theta), & \mbox{ if } \lambda\in \R_+^m \\
		  							  -\infty, & \mbox{otherwise.}
						\end{cases} 	
\end{align*}
When $\rho\geq 0$,
the dual problem %corresponding to
of ${\cal C}(\theta)$
is defined as %follows:
\begin{align*}
\tag{$D_{\rho}$}
\max_{\lambda \in \R_+^m}\quad \left\{g_\rho(\lambda;\theta)\triangleq\inf_{x\in X} {\cal L}_\rho(x,\lambda;\theta)\right\}.
\end{align*}
%where %$g_{\rho}$ is defined as follows:%
%$g_\rho(\lambda;\theta)\triangleq\inf_{x\in X} {\cal L}_\rho(x,\lambda;\theta).$
%The following results
{Throughout, we assume the following:{
\begin{assumption}\label{Assump: Aug_1}
\begin{enumerate}
\item[]
\item[(i)] The functions %\sout{$p$, $q$}
$f(x,\theta)$ and $h_i(x,\theta)$ are convex in $x\in X$ for all $\theta
	\in \Theta$ for $i = 1, \hdots, m$ and $X \subseteq\Real^n$ and $\Theta$ are convex compact sets.
%
%\item[(ii)] \sout{$\nabla_x p(x;\theta)$ is uniformly Lipschitz continuous in $x$ for all $\theta \in \Theta$ with Lipschitz constant $L_p$.}
%\item[(ii)] The set %\sout{${\cal D}_q$}
%$X$ is compact and convex.
%\item[(iii)] The Jacobian of $h$ w.r.t. $\theta$, denoted by
%$\J_{\theta} h(x;\theta)\in\R^{m\times d}$, exists and is uniformly
%bounded for all $(x,\theta)\in X\times\Theta$; %\sout{in $\theta$ for all
	%\ha{$x\in X\cap {\cal D}_q$ {\bf should be removed}};
  %   i.e, %\sout{there exists a constant $L_{h,\theta}$}
    %$\exists M_{h}>0$ such that $\|\J_{\theta} h(x;\theta)\|\leq M_{h}$ all $(x,\theta)\in X\times\Theta$.
%\sout{for all  $x\in X\cap {\cal D}_q$ and all $\theta\in \Theta$.}
%\item[(iv)] For $i=1,\hdots,m$, the functions $h_i(x;\theta)$ are
	%Lipschitz continuous in $\theta$ over $\Theta$ for all $x\in X$
	%{with Lipschitz constant $L^i_h$.}
\item[(ii)] The function $f(x;\theta)$ is Lipschitz continuous in $\theta$ over $\Theta$ for all
	$x\in X$ {with constant $L_f$}; i.e. {for all $x \in X$}, $\|f(x;\theta_1)-f(x;\theta_2)\|\leq L_{f}\|\theta_1-\theta_2\|$ for all $\theta_1,\theta_2\in \Theta$.
\item[(iii)] {$h(x;\theta)$ is an affine map
in $x$ for every $\theta \in \Theta$, i.e.,
	$h(x;\theta)=A(\theta)x+b(\theta)$ for some $A(\theta)\in\R^{m\times
		n}$ and $b(\theta)\in\R^m$. Suppose $A(\theta)$ and $b(\theta)$
		are Lipschitz continuous in $\theta$. Hence, $h(x;\theta)$ is
		Lipschitz continuous in $\theta$ with constant $L_h$
			uniformly for all $x\in X$. Clearly, $h(B(0,1);\theta)
				\subseteq B(b(\theta), \sigma_{\max}(A(\theta)))$ for
					all $\theta \in \Theta$, where $B(\bar{y},r):=\{y:\
					\|y-\bar{y}\|\leq r\}$. Hence, there is a constant
					$\sigma_h$ s.t. $h(B(0,1);\theta) \subseteq
					\sigma_h B(0, 1)$ for all $\theta \in \Theta$, since $\sigma_{\max}(A(\theta))$ is continuous in $\theta$ and $\Theta$ is compact.}
\item[(iv)] $X^*(\lambda;\theta)$ is pseudo-Lipschitz
%\us{HA-NEED TO DEFINE}
in $\theta$ uniformly
in $\lambda$ with constant $\kappa_X$, where $ X^*(\lambda;\theta) = \mbox{arg}\min_{x \in \cal X} {\cal
	L}_0(x,\lambda;\theta), $
	i.e., for any $\theta_1,\theta_2\in\Theta$,
	$X^*(\lambda;\theta_1)\subseteq X^*(\lambda;\theta_2)+\kappa_X B(0,1)$ for all $\lambda\geq 0$.
\end{enumerate}
\end{assumption}
}

%\noindent {\bf Remark}: Under Assumption~\ref{Assump: Aug_1}(iv), it
%can be shown
%that if $L_{h}\triangleq {\displaystyle \max_{i \in \{1,\hdots,m\}}}
%\{L^i_{h}\}$, then for all $x\in X$ and $\theta_1,\theta_2\in \Theta$,
%$\|h(x;\theta_1)-h(x;\theta_2)\|\leq L_{h}\|\theta_1-\theta_2\|.$
%Furthermore, under Assumption~\ref{Assump: Aug_1}(vi),	
Rather than focusing on the nature of the algorithm employed for resolving the
learning problem, instead we impose a requirement that the adopted scheme
produces a sequence that converges to the optimal solution at a non-asymptotic
linear rate.
\begin{assumption}
\label{Assump: learn_lin_rate}
There exists a learning scheme that generates a sequence $\{\theta_k\}$
such that $\theta_k \to \theta^*$ at a linear rate as $k\to \infty$, i.e., there
exists a constant $q_\ell\in(0,1)$ such that for all $k \geq 0$ and $\theta_0
\in \Theta$, one has $\|\theta_k-\theta^*\|\leq q_\ell^k\|\theta_0-\theta^*\|$.
In addition, at iteration $k$ of the optimization problem ${\cal C}$, only $\theta_1,\dots,\theta_k$ are revealed.
\end{assumption}
}
Lemma~\ref{Lemma: Aug_Rock_Lemmas}, pertaining to various
properties of the gradient of the dual function
$\nabla_{\lambda}g_\rho$, will be used in our analysis. The proof of Lemma~\ref{Lemma: Aug_Rock_Lemmas} can be
found in~\cite{rockafellar1973dual} and is omitted here.
\begin{lemma}\label{Lemma: Aug_Rock_Lemmas}
Suppose Assumption~\ref{Assump: Aug_1} holds. %Then, the following hold.
\begin{enumerate}
\item[(i)] For any $\rho>0$ and $\theta \in \Theta$, the dual function $g_\rho (\lambda;\theta)$ is {\it everywhere} finite, continuously differentiable concave function over $\R^m$; more precisely, $g_\rho(\lambda;\theta)=\max_{w\in\R^m}\{g_0(w;\theta)-\tfrac{1}{2\rho}\|w-\lambda\|^2\}$, i.e., $g_\rho(\cdot,\theta)$ is the Moreau regularization of $g_0(\cdot,\theta)$ for all $\theta\in\Theta$. Therefore,
	$\nabla_\lambda g_\rho(\lambda;\theta)$ is Lipschitz continuous in
	$\lambda$ with constant $\tfrac{1}{\rho}$.
\item[(ii)] For any given $\lambda\in\R^m_+$ and $\theta\in \Theta$, $\nabla_\lambda g_\rho$ can be computed as $\nabla_\lambda g_\rho (\lambda;\theta)=\nabla_{\lambda} {\cal L}_\rho (x^*(\lambda),\lambda;\theta)$, where $x^*(\lambda)\in \argmin_{x\in X} {\cal L}_\rho
(x,\lambda;\theta)$.
\item[(iii)] Given $\lambda\in\R^m_+$ and $\theta\in \Theta$,
	suppose $\tilde{x}(\lambda)$ is an {\it inexact} solution to $\min_{x\in X}{\cal L}_\rho(x,\lambda;\theta)$
	with accuracy $\alpha$, i.e.,  $\tilde{x}(\lambda)\in X$ satisfies
${\cal L}_\rho(\tilde{x}(\lambda),\lambda;\theta)\leq g_\rho(\lambda;\theta)+\alpha,$
then \vspace*{-3mm}
$$\|\nabla_{\lambda} {\cal L}_\rho(\tilde{x}(\lambda),\lambda;\theta)-\nabla_{\lambda} g_\rho(\lambda;\theta)\|^2\leq 2\alpha/\rho.$$
\end{enumerate}
\end{lemma}
Next, we examine %the Lipschitzian property of
the continuity of $\nabla_{\lambda}
g_{\rho}(\lambda;\theta)$ in $\theta\in\Theta$.

\begin{lemma}[\textbf{Lipschitz continuity of $\nabla_{\lambda} g_\rho$ in $\theta\in\Theta$}]
\label{Lemma:Aug_Lips_cont_grad_g}
Suppose Assumption~\ref{Assump: Aug_1} holds. Then, we have that $\nabla_{\lambda}
g_\rho(\lambda;\theta)$ is Lipschitz continuous in $\theta$ over $\Theta$ uniformly in $\lambda\in\R^m$ with constant $\kappa_h+\kappa_X \sigma_{h}$.
\end{lemma}
\begin{proof}
Due to limited space, we omit the proof. For details, see Proposition 2.4 in the extended version this paper~\cite{1510.00490}.
\end{proof}
\noindent {\bf Remark:} We now comment on the conditions under which $X^*(\lambda;\theta)$
is pseudo-Lipschitz in $\theta$. When $f(x;\theta)$ is a \emph{differentiable}
convex function in $x$ for every $\theta$, and $h(x;\theta)$ is an
affine function in $x$ for every $\theta$, then $X^*(\lambda;\theta)$ is
the solution set of VI$(X, \nabla_{x} {\cal
		L}_0(.,\lambda;\theta))$ when $\lambda \in \Real^m_+$ and
$\theta \in \Theta$. We consider two sets of problem classes in
providing conditions under which the associated solution sets admit
pseudo-Lipschitzian properties: {\em 1) Parametrized quadratic
	programming:} If $f(x;\theta)$ is a quadratic function for every
	$\theta \in \Theta$ and $X$ is a polyhedral set, the mapping of the
	variational inequality problem is affine; such a problem  is
	generally referred to as an affine variational inequality problem
	and denoted by AVI$(X,M(\theta),q(\theta))$
where $M(\theta) x + q(\theta) = \nabla_x f(x;\theta) + %\sum_{i=1}^m \lambda_i \nabla_x h_i(x;\theta)
A(\theta)^\top \lambda$ while its solution set is denoted by
SOL$(X,M(\theta),q(\theta))$, $\mbox{int}(K^+)$ denotes the
	interior of the positive dual cone of $K$, and $K^+ \triangleq \{y:
		y^Tz \geq 0, \forall z \in K\}$. Then under Theorem 7.4~\cite{lee2010quadratic}, if
$M(\theta)$ is positive semidefinite over $X$ for all $\theta \in
\Theta$, and if $q(\theta) \in \mbox{int} \left( \left[\mbox{SOL} (X, M(\theta),
			0)\right]^+\right)$,
%\ser{COMMENT: this condition is not clear? What do you mean by $\left[\mbox{SOL} (X, M(\theta), 0)\right]^+$?}
then there exists scalars $\epsilon$
and $\kappa$
%\ser{do they depend on the given $\theta$?}
 such that if $ \max\limits_{\thetahat \in \Theta} \left\{ \|M(\thetahat) - M(\theta)\|, \|
q(\thetahat) - q(\theta) \|\right\}< \epsilon, $
then $X^*(\lambda;\thetahat) \subseteq X(\lambda;\theta) + \kappa
\left(\|M(\thetahat)-M(\theta)\| + \|q(\thetahat)-q(\theta)\|\right) B(0,1).$
Under a compactness assumption on $\Theta$, this ``local'' Lipschitzian result can be globalized.
{\em 2) Parametrized convex programming:} More generally, suppose  $f(x;\theta)$ is a nonlinear convex function and $B(H;\epsilon,S)$ denotes an $\epsilon-$neighborhood of $H$ containing all continuous functions $G$ that are within $\epsilon$ distance to $H$ when restricted to the set $S$, i.e.,
$$\| G-H\|_S \triangleq \sup_{y \in S} \| G(y) - H(y)\| < \epsilon. $$Then
 we define the associated  VI$(X,\nabla_{x} {\cal L}(.,\lambda;\theta))$  as
{\em semi-stable} if there exist two positive scalars $c$ and $\epsilon$
%\ser{depending on $\theta$?}
such that for every $\nabla_{x} {\cal L}(.,\lambda;\thetahat) \in B(\nabla_{x}{\cal
L}(.,\lambda;\theta);\epsilon, X)$, we have that
\begin{align*}
 &X^*(\lambda;\thetahat) \subseteq X^*(\lambda;\theta) \\&+ c \sup_{x \in
X} \| \nabla_{x} {\cal L}(x,\lambda;\thetahat) - \nabla_{x} {\cal
				 L}(x,\lambda;\theta)\| B(0,1).
\end{align*}
In fact, a necessary and sufficient condition for semi-stabililty of VI$(X,F)$ is the following~\cite[Prop.~5.5.5]{facchinei02finite}: There exists two positive scalars $c$ and $\epsilon$, such that for all $q \in \Real^n$,
$$ \ \|q \| \ < \ \epsilon \ \implies \mbox{SOL}(X,q+F) \subseteq \mbox{SOL}(X,F) + B(0,c\|q\|).$$

We conclude this section by presenting the misspecified variant of the inexact augmented Lagrangian scheme with constant penalty $\rho>0$. Notably, if $\theta_k = \theta^*$ for all $k\geq 0$, this reduces to the traditional version considered in~\cite{rockafellar1973dual}. %which is correctly specified.
		\vspace{-0.05in}
\begin{algorithm}
\small
\caption{{Misspecified inexact aug. Lag. scheme}}
Given $\lambda_0=\mathbf{0}\in\R^m$, and $\rho>0$, let $\{\alpha_k\}$,$\{\theta_k\}$ be given sequences. Then for all $k\geq 0$,
\label{Alg: Aug_const_rho}
\begin{enumerate}
\item[(1)]  find $x_k$ such that ${\cal L}_\rho(x_k,\lambda_k;\theta_{k})\leq g_\rho(\lambda_k;\theta_{k})+\alpha_k$;
\item [(2)] $\lambda_{k+1}=\lambda_k+\rho\nabla_{\lambda} {\cal L}_\rho(x_k,\lambda_k,\theta_{k})$;
\item [(3)] $k:=k+1$;
\end{enumerate}
\end{algorithm}
		\vspace{-0.1in}
%We make the following assumption on the sequence $\{\alpha_k\}$.
\begin{assumption}
\label{Assump:Aug_2} $\{\alpha_k\}$ is chosen such that
{$\sum_{k=0}^{\infty} \sqrt{\alpha_k} <\infty$}.
\end{assumption}
\vspace*{2mm}
Under this assumption, we show \emph{(i)}
	$f^*-g_\rho(\bar{\lambda}_k)\leq\mathcal{O}(1/k)$ for
		$\bar{\lambda}_k\triangleq {1\over k}{\sum_{i=1}^k \lambda_i}$,
		\emph{(ii)} $d_{\R^m_-}(h(\bar{x}_k;\theta^*))\leq
			\mathcal{O}(1/\sqrt{k})$, and \emph{(iii)}
	$-\mathcal{O}(1/\sqrt{k})\leq f(\bar{x}_k;\theta^*)-f^*\leq
		\mathcal{O}(1/k)$ for $\bar{x}_k\triangleq {1\over
			k+1}{\sum_{i=0}^k x_i}$. After proving these bounds
			independently, we became aware of related recent
			work~\cite{1506.05320 }, where
    %the perfectly specified variant of
			Algorithm~\ref{Alg: Aug_const_rho} is considered with $\alpha_k=\alpha>0$ for all $k\geq 0$,
			assuming \emph{perfect information}, i.e., $\theta_k=\theta^*$ for all $k\geq 0$.
			In~\cite{1506.05320 }, it is shown that \emph{(i)}
	$f^*-g_\rho(\bar{\lambda}_k)\leq\mathcal{O}(1/k)+\alpha$,
		\emph{(ii)} $d_{\R^m_-}(h(\bar{x}_k;\theta^*))\leq
			\mathcal{O}(1/\sqrt{k})$, and \emph{(iii)}
	$-\mathcal{O}(1/\sqrt{k})\leq f(\bar{x}_k;\theta^*)-f^*\leq
		\mathcal{O}(1/k)+\alpha$. Therefore, according
		to~\cite{1506.05320 },  $\alpha$ should be fixed as a small
		constant in accordance with the desired accuracy. %as it appears in \emph{both} primal and dual suboptimality bounds.
		Since $\alpha$ is fixed in~\cite{1506.05320}, such avenues
			can, at best, provide approximate solutions.
	%suboptimality of the iterate sequence may \emph{stall} after certain iterations.
	In contrast, our
		method may  start with large $\alpha_0$ and gradually
		decrease it, ensuring \emph{both} numerical stability
		and asymptotic convergence to optimality.
%in contrast with~\cite{1506.05320} where the scheme provides approximate solutions at best.
		\vspace{-0.05in}
%-----------------------------------------------
\section{Rate of convergence analysis}\label{Sec:III}
%-----------------------------------------------
		\vspace{-0.01in}
We begin by showing that dual variables stay bounded by using
a supporting Lemma whose proof follows from
Lemma~\ref{Lemma: Aug_Rock_Lemmas}(i) and the properties of proximal
maps (cf.~\cite{hiriart2001convex}). %\red{in~\cite{Rockafellar73_1J}}.
\begin{lemma}
\label{Lemma:Non_exp_pi}
Let $\pi_\rho(\lambda;\theta):=\argmax_{w\in\R^m}g_0(w;\theta)-\tfrac{1}{2\rho}\|w-\lambda\|^2$ for $\theta\in\Theta$, i.e., the proximal map of $g_0(\cdot;\theta)$. Then $\pi_\rho(\lambda;\theta)=\lambda+\rho \nabla_\lambda g_\rho(\lambda;\theta)$, and $\pi_\rho$ is {\it nonexpansive} in $\lambda$ for all $\theta\in \Theta$.
\end{lemma}
\begin{theorem}[\textbf{Boundedness of $\{\lambda_k\}$}]
\label{Theorem: Aug_lambda_bnd_cont_rho}
Let Assumptions~\ref{Assump: Aug_1}--\ref{Assump:Aug_2} hold, and $\lambda^*$ be an arbitrary solution to the Lagrangian dual of ${\cal C}(\theta^*)$, i.e., $\lambda^*\in\argmax_\lambda g_0(\lambda;\theta^*)$. Then for all $k\geq 1$, $\|\lambda_k-\lambda^*\|\leq C_\lambda$, where $C_\lambda$ is defined as follows:
\begin{equation}
C_\lambda \triangleq\sqrt{2\rho}\sum_{i=0}^\infty\sqrt{\alpha_i}+\rho M_{h}{\|\theta_0-\theta^*\|\over 1-q}+\|\lambda^*\|.\label{bd-C-lambda}
\end{equation} %+\|\lambda^*\|.$$
\end{theorem}
\begin{proof}
We begin by deriving a bound on $\|\lambda_{k+1}-\pi_{\rho}(\lambda_k;\theta_k)\|$ by utilizing the definition of $\lambda_{k+1}$ from Step 2 in Algorithm~\ref{Alg: Aug_const_rho}:
\begin{align}
& \quad \ \|\lambda_{k+1}-\pi_\rho(\lambda_k;\theta_{k})\| \notag \\
&=\|\lambda_{k}+\rho\nabla_{\lambda} {\cal L}_\rho(x_k,\lambda_k;\theta_{k})-\lambda_k-\rho\nabla_\lambda g_{\rho}(\lambda_k;\theta_{k})\|\notag\\
&= \rho\|\nabla_{\lambda} {\cal L}_\rho(x_k,\lambda_k;\theta_{k})-\nabla_\lambda g_{\rho}(\lambda_k;\theta_{k})\|\leq \sqrt{2\rho\alpha_k},\label{proof:Aug_bnd_lambda_1}
\end{align}
where the last inequality follows from Lemma~\ref{Lemma: Aug_Rock_Lemmas} (iii). Since $g_\rho(\cdot;\theta^*)$ is the Moreau regularization of $g_0(\cdot;\theta^*)$, it is true that $\lambda^*\in \argmax_{\lambda} g_{\rho}(\lambda,\theta^*)$ for all $\rho>0$. Hence, $\nabla_\lambda g_\rho(\lambda^*;\theta^*)=0$ and $\lambda^*=\pi_{\rho}(\lambda^*,\theta^*)$. From this observation, we obtain the bound below:
%\begin{align}
\begin{align}
&\|\pi_\rho(\lambda_k,\theta_{k})-\lambda^*\|
 =  \|\pi_\rho(\lambda_k,\theta_k)-\pi_{\rho}(\lambda^*,\theta^*)\|\notag\\
&\leq \|\pi_\rho(\lambda_k,\theta_{k})-\pi_\rho(\lambda_k,\theta^*)\|+\|\pi_\rho(\lambda_k,\theta^*)-\pi_\rho(\lambda^*,\theta^*)\|
\notag\\
&= \rho\|\nabla_{\lambda} g_{\rho}(\lambda_k,\theta_{k})-\nabla_{\lambda} g_{\rho}(\lambda_k,\theta^*)\| \notag\\
& +\|\pi_\rho(\lambda_k,\theta^*)-\pi_\rho(\lambda^*,\theta^*)\notag \|
\\ & \leq \rho M_{h}\|\theta_{k}-\theta^*\|+\|\lambda_k-\lambda^*\|.\label{proof:Aug_bnd_lambda_2}
\end{align}
%\end{align}
%where the first inequality follows from the triangle
%inequality, the second equality is a result of invoking the definition of
%$\pi_{\rho}(\lambda, \theta)$,   and the last inequality
This follows from the Lipschitz continuity of $\nabla_{\lambda} g_{\rho}$ and the nonexpansivity of $\pi_\rho$ in
$\lambda$ (Lemma~\ref{Lemma:Non_exp_pi}). Hence, from
~\eqref{proof:Aug_bnd_lambda_1} and ~\eqref{proof:Aug_bnd_lambda_2}, we obtain for all $i\geq 0$ that
\begin{equation*}
\begin{array}{c}
\|\lambda_{i+1}-\lambda^*\| %&\leq\|\lambda_{k+1}-\pi_\rho(\lambda_k,\theta_{k})\|+\|\pi_\rho(\lambda_{k},\theta_k)-\lambda^*\|\\
\leq \sqrt{2\rho\alpha_i}+\rho M_{h}\|\theta_{i}-\theta^*\|+\|\lambda_i-\lambda^*\|.
\end{array}
\end{equation*}
For $k\geq 0$, by summing the above inequality over $i=0,\ldots,k$, and using the fact that $\lambda_0=\mathbf{0}$, we get
\begin{align*}
\|\lambda_{k+1}-\lambda^*\|&\leq \sum_{i=0}^k\left(\sqrt{2\rho\alpha_i}+\rho M_{h}\|\theta_{i}-\theta^*\|\right)+\|\lambda_0-\lambda^*\| \notag\\
&\leq \sqrt{2\rho}\sum_{i=0}^\infty\sqrt{\alpha_i}+\rho M_{h}{\|\theta_0-\theta^*\|\over 1-q}+\|\lambda^*\|.
\end{align*}
%Finally, for all $k$, we have that the following holds:
%\begin{align} \notag
%& \quad \|\lambda_{k+1}\|\leq \|\lambda_{k+1}-\lambda^*\|+\|\lambda^*\|\\
%		\notag
%&\leq \sqrt{2\rho}\sum_{i=0}^\infty\sqrt{\alpha_i}+\rho M_{h,\theta}{\|\theta_0-\theta^*\|\over 1-q_\ell}+\|\lambda_0-\lambda^*\|+\|\lambda^*\|\\
%&\triangleq C_\lambda.
%\end{align}
\end{proof}
\vspace{-.1 cm}
\noindent {\bf Remark}: It is worth emphasizing that the bound $C_{\lambda}$ can be tightened when $\theta^*$ is known, i.e., since $\theta_0 = \theta^*$, the second term disappears.

Next, we prove that the augmented Lagrangian scheme generates a sequence
$\{\lambda_k\}$ such that $\bar{\lambda}_k \to \lambda^*$ as $k \to \infty$ by
deriving a rate statement on the ergodic average sequence.
\begin{theorem}[\textbf{Bound on dual suboptimality}]
\label{Theorem:Aug_bnd_dual_sub}
Let Assumptions~\ref{Assump: Aug_1} -- \ref{Assump:Aug_2} hold and let $\{\lambda_k\}_{k\geq 1}$ denote the sequence generated by Algorithm~\ref{Alg: Aug_const_rho}. In addition, let $\bar{\lambda}_k\triangleq {1\over k}{\sum_{i=1}^k \lambda_i}$. Then it follows that for all $k\geq 1$:
\begin{align}
f^*-g_\rho(\bar{\lambda}_{k};\theta^*)=\sup_{\lambda} \ g_\rho(\lambda;\theta^*)-g_\rho(\bar{\lambda}_{k};\theta^*)&\leq{B_g\over k},\label{Bound_dual_cons_rho}
\end{align}
where $\lambda^*\in \argmax g_0(\lambda,\theta^*)$, $C_\lambda$ is
defined in Theorem~\ref{Theorem: Aug_lambda_bnd_cont_rho}, and $B_g$ is defined as follows:
$$B_{g}\triangleq \tfrac{1}{2\rho}\|\lambda^*\|^2+C_{\lambda}\left(\sqrt{\tfrac{2}{\rho}}\sum_{k=0}^{\infty} \sqrt{\alpha_k}+\frac{M_{h}\|\theta_0-\theta^*\|}{1-q}\right).$$
\end{theorem}
\begin{proof}
Note that from Lemma~\ref{Lemma: Aug_Rock_Lemmas} and using the fact that the duality gap for $\mathcal{C}(\theta^*)$ is 0, it follows that $f^*=\max_\lambda g_\rho(\lambda;\theta^*)$ for all $\rho>0$. Using the Lipschitz continuity of $\nabla_{\lambda} g_\rho(\lambda,\theta^*)$ in $\lambda$ with constant ${1/\rho}$, for $i\geq 0$, we get
\begin{align}
-g_\rho(\lambda_{i+1};\theta^*)&\leq -g_\rho(\lambda_i;\theta^*)-\nabla_{\lambda} g_\rho(\lambda_i;\theta^*)^\top(\lambda_{i+1}-\lambda_i)\notag\\
&\quad+\tfrac{1}{2\rho}\|\lambda_{i+1}-\lambda_i\|^2.\label{proof:ineq_decent}
\end{align}
Under the concavity of $g_\rho(\lambda;\theta^*)$ in $\lambda$, we have that
$$-g_\rho(\lambda^*;\theta^*)\geq -g_\rho(\lambda_i;\theta^*)-\nabla_{\lambda} g_\rho(\lambda_i;\theta^*)^T(\lambda^*-\lambda_i) .$$
By combining the above inequality and \eqref{proof:ineq_decent}, we get
\begin{align}
& \quad -g_\rho(\lambda_{i+1};\theta^*) \notag \\
%& \leq -g_\rho(\lambda_i;\theta^*)-\nabla_{\lambda} g_\rho(\lambda_i;\theta^*)^T(\lambda_{i+1}-\lambda_i)\notag+{1\over 2\rho}\|\lambda_{k+1}-\lambda_i\|^2 \\
& \leq -g_\rho(\lambda^*;\theta^*)-\nabla_{\lambda} g_\rho(\lambda_i;\theta^*)^T(\lambda_{i+1}-\lambda^*)+\tfrac{1}{2\rho}\|\lambda_{i+1}-\lambda_i\|^2  \notag \\
%\notag \\
		%& -\nabla_{\lambda} g_\rho(\lambda_i;\theta^*)^T(\lambda_{i+1}-\lambda_i)+{1\over 2\rho}\|\lambda_{k+1}-\lambda_i\|^2  \notag \\
%& = -g_\rho(\lambda^*;\theta^*)-\nabla_{\lambda} g_\rho(\lambda_i;\theta^*)^T(\lambda_{i+1}-\lambda^*)+{1\over 2\rho}\|\lambda_{i+1}-\lambda_i\|^2 \notag \\
& \notag =-g_\rho(\lambda^*;\theta^*)-\nabla_{\lambda} {\cal
	L}_\rho(x_i,\lambda_i;\theta_{i})^T(\lambda_{i+1}-\lambda^*)\notag\\
&\quad+\delta_i^T(\lambda_{i+1}-\lambda^*)+s_i^T(\lambda_{i+1}-\lambda^*)+\tfrac{1}{2\rho}\|\lambda_{i+1}-\lambda_i\|^2
	\notag \\
&\notag \leq-g_\rho(\lambda^*;\theta^*)-\tfrac{1}{\rho}
(\lambda_{i+1}-\lambda_{i})^T(\lambda_{i+1}-\lambda^*)+\tfrac{1}{2\rho}\|\lambda_{i+1}-\lambda_i\|^2\\
&\quad+\|\delta_i\|\|\lambda_{i+1}-\lambda^*\|+\|s_i\|\|\lambda_{i+1}-\lambda^*\|,\label{ineq-bd-lamb}
\end{align}
where {$\delta_i\triangleq \nabla_{\lambda} g_\rho(\lambda_i;\theta_{i})-\nabla_{\lambda} g_\rho(\lambda_i;\theta^*)$} and {$s_i\triangleq \nabla_{\lambda} {\cal L}_\rho(x_i,\lambda_i;\theta_{i})-\nabla_{\lambda} g_\rho(\lambda_i;\theta_{i})$}.
By noting that
$\|\lambda_{i+1}-\lambda_i\|^2+2(\lambda_{i+1}-\lambda_{i})^T(\lambda^*-\lambda_{i+1})=\|\lambda_i-\lambda^*\|^2-\|\lambda_{i+1}-\lambda^*\|^2$,
	we can rewrite \eqref{ineq-bd-lamb} as
\begin{align} \notag
-g_\rho(\lambda_{i+1};\theta^*)&\leq-g_\rho(\lambda^*;\theta^*)+\left(\|\delta_i\|+\|s_i\|\right)\|\lambda_{i+1}-\lambda^*\|\\
&\quad+\tfrac{1}{2\rho}\left(\|\lambda_i-\lambda^*\|^2-\|\lambda_{i+1}-\lambda^*\|^2\right).\label{bd-lamb-ineq2}
\end{align}
By summing \eqref{bd-lamb-ineq2} over $i=0,\ldots,k-1$, replacing
$g_\rho(\lambda^*;\theta^*)$ by $f^*=\sup_\lambda g_\rho(\lambda,\theta^*)$ and setting $\lambda_0=\mathbf{0}$, we obtain
\begin{align}
&-\sum_{i=0}^{k-1} \Big(g_\rho(\lambda_{i+1};\theta^*)-%\sup_\lambda g_\rho(\lambda;\theta^*)
f^*\Big)+\tfrac{1}{2\rho} \|\lambda_{k}-\lambda^*\|^2\notag\\
&\leq\tfrac{1}{2\rho} \|\lambda^*\|^2+\sum_{i=0}^{k-1}\left(\|\delta_i\|+\|s_i\|\right)\|\lambda_{i+1}-\lambda^*\|.
\label{proof: bnd_dual_sub_opt}
\end{align}
Under concavity of $g_\rho(\lambda;\theta^*)$ in $\lambda$, the following holds:
\begin{align*}
- \Big(g_\rho(\bar{\lambda}_{k};\theta^*)-%\sup_{\lambda} g(\lambda,\theta^*)
f^*\Big)\leq -{1\over k}\sum_{i=0}^{k-1} \Big(g_\rho(\lambda_{i+1};\theta^*)-%\sup_\lambda g_\rho(\lambda,\theta^*)
f^*\Big).
 \end{align*}
By dividing both sides of \eqref{proof: bnd_dual_sub_opt} by $k$ and dropping the positive term on the left hand side, we get
\begin{align} \notag
&\quad \ f^* %\sup_\lambda g_\rho(\lambda;\theta^*)
-g_\rho(\bar{\lambda}_{k};\theta^*)\\
& \leq\frac{1}{k}\left(\tfrac{1}{2\rho}\|\lambda^*\|^2 +\sum_{i=0}^{k-1}\left(\|\delta_i\|+\|s_i\|\right)\|\lambda_{i+1}-\lambda^*\|\right).\notag
%& \leq\frac{1}{k+1}\left(\tfrac{1}{2\rho}\|\lambda_{0}-\lambda^*\|^2 +\sum_{i=0}^\infty\left(\|\delta_i\|+\|s_i\|\right)\|\lambda_{i+1}-\lambda^*\|\right) \notag
%\label{proof:averg_bnd}
\end{align}
Lemma~\ref{Lemma: Aug_Rock_Lemmas} and Lemma~\ref{Lemma:Aug_Lips_cont_grad_g} imply that $\|s_i\|\leq\sqrt{
	{2\alpha_i\over \rho}}$, and
$\|\delta_i\|\leq M_{h}\|\theta_{i}-\theta^*\|$, resp., for all $i\geq 0$. In addition, from Theorem~\ref{Theorem: Aug_lambda_bnd_cont_rho}, we have $\|\lambda_i-\lambda^*
\|\leq C_{\lambda}$ for all $i\geq 1$. Then by the summability of $\sqrt{\alpha_i}$, we have that
\begin{eqnarray}
\lefteqn{\sum_{i=0}^{\infty} (\|\delta_i\|+\|s_i\|)\|\lambda_{i+1}-\lambda^*\|}\notag\\
 &\leq & C_{\lambda}\left(M_h \sum_{i=0}^{\infty} \|\theta_{i}-\theta^*\|+\sqrt{\tfrac{2}{\rho}}\sum_{i=0}^\infty\sqrt{\alpha_i}\right).\label{proof:bnd_subopt_1}
\end{eqnarray}
Furthermore, substituting $\sum_{i=0}^{\infty} \|\theta_{i}-\theta^*\|= \|\theta_0-\theta^*\|/(1-q)$ into \eqref{proof:bnd_subopt_1} gives the desired bound and completes the proof.
%\begin{align}
%\sum_{i=0}^{\infty} \|\delta_i\|\|\lambda_{i+1}-\lambda^*\|
%&\leq \sum_{i=0}^{\infty}
%C_{\lambda}M_{h,\theta}\|\theta_{i}-\theta^*\|\notag \\
%	& =C_\lambda M_{h,\theta}{\|\theta_0-\theta^*\| \over 1-q_\ell}.
%\label{proof:bnd_subopt_2}
%\end{align}
%Substituting ~\eqref{proof:bnd_subopt_1} and ~\eqref{proof:bnd_subopt_2}
%into ~\eqref{proof:averg_bnd}, the required bound is obtained.
\end{proof}

Next, we derive a bound on the primal {\it infeasibility}, where the primal
iterate sequence is computed such that Step~1 in Algorithm~\ref{Alg:
	Aug_const_rho} is satisfied. Prior to proving our main result, we provide some
	supporting technical lemmas.
\begin{lemma}
\label{Lemma:Aug_grad_bnd}
Assume that $\phi(\lambda):\R^m\to \R$ is a concave function whose
supremum is finite and is attained at $\lambda_{\phi}^*$. In addition,
 assume that $\nabla \phi$ is Lipschitz continuous with constant
 $L_{\phi}$. Then, for all $\lambda\in \R^m$, we have that
$\|\nabla \phi(\lambda)\|\leq \sqrt{2L_{\phi}\big(\phi(\lambda_{\phi}^*)-\phi(\lambda)\big)}.$
\end{lemma}
This is an immediate result of Theorem 2.1.5 in~\cite{nesterov2013introductory}.
%\begin{proof}
%\ha{
%}
%\begin{comment}
%Lipschitz continuity of $\nabla \phi$ implies that
%\begin{align*}
%\phi(\lambda')\geq \phi(\lambda)+\nabla_{\lambda} \phi(\lambda)^\top(\lambda'-\lambda)-{L_{\Phi}\over 2}\|\lambda'-\lambda\|^2.
%\end{align*}
%for all $\lambda,\lambda'\in \R^m$. Taking the supremum over $\lambda$ from both sides of above inequality proves the claim after rearranging the terms, i.e.,
%\begin{align*}
%\phi(\lambda^*_{\phi})&\geq \phi(\lambda)+\sup_{\lambda' \in \R^m}\{\nabla \phi(\lambda)^T(\lambda'-\lambda)-{L_{\phi}\over 2}\|\lambda'-\lambda\|^2\}\\
%%&=\phi(\lambda)+\sup_{u \in \R^m}\{\nabla \phi(\lambda)^Tu-{L_{\phi}\over 2}\|u\|^2\}\\
%&=\phi(\lambda)+{\|\nabla \phi(\lambda)\|^2\over 2L_{\phi}}.
%\end{align*}
%\end{comment}
%\end{proof}
Next, we derive a bound on $d_{\R_+^m}\big(y+y'\big)$ for any $y,
	y' \geq 0$.
\begin{lemma}
\label{Lemma:d_triangle}
For all $y,y'\in \R_{+}^m$, $d_{\R_+^m}\big(y+y'\big)\leq d_{\R_+^m}\big(y\big)+\|y'\|.$
\end{lemma}
\begin{proof}
The result immediately follows from the definition $d_{\R_+^m}$ and the non-expansivity of $\Pi^c_{\R_+^m}(x)\triangleq x-\Pi_{\R_+^m}(x)$.
%the triangle
%inequality, we obtain the following inequality:
%\begin{align*}
%d_{\R_+^m}\big(y+y'\big) & =\|y+y'-\Pi_{\R_+^m}(y+y')\|\\
%		& =\left\|y-\Pi_{\R_+^m}(y)\right. \\
%		& \left. +y'+y-y+\Pi_{\R_+^m}(y)-\Pi_{\R_+^m}(y'+y)\right\|\\
%&\leq \|y-\Pi_{\R_+^m}(y)\|\\
%	& +\|y'+y-\Pi_{\R_+^m}(y'+y)-(y-\Pi_{\R_+^m}(y))\|.
%\end{align*}
%Define $\Pi^c_{\R_+^m}(x)\triangleq x-\Pi_{\R_+^m}(x)$, which can be shown easily that it is a nonexpansive operator. Using this operator, we can rewrite the above inequality as follows:
%\begin{align*}
%d_{\R_+^m}\big(y+y'\big)& \leq \|y-\Pi_{\R_+^m}(y)\|+\|\Pi^c_{\R_+^m}(y+y')-\Pi^c_{\R_+^m}(y)\|\\
%&\leq d_{\R_+^m}\big(y\big)+\|y'\|,
%\end{align*}
%where the last inequality follows from nonexpansivity of $\Pi^c_{\R_+^m}$.
\end{proof}
We now derive the bound on the primal infeasibility. %after $k$ iterations.
\begin{theorem}[\textbf{Bound on primal infeasibility}]
\label{Theorem:Aug_prim_infeas_cons_rho}
Let Assumptions~\ref{Assump: Aug_1}--\ref{Assump:Aug_2} hold and let
$\{\lambda_k\}_{k\geq 0}$ and $\{x_k\}_{k\geq 0}$ denote the sequences generated by
Algorithm~\ref{Alg: Aug_const_rho}. Furthermore, let $\xbar_k={1\over k+1}{\sum_{i=0}^{k} x_i}$. Then, it follows that
\begin{align}
d_{\R_-^m}\Big(h(\xbar_k,\theta^*)\Big)\leq{\cal V}(k)\triangleq \frac{C_1}{\sqrt{k+1}}+\frac{C_2}{k+1},\label{Def:V(k)}
\end{align}
where %${\cal V}(k)$ is defined as
$C_1:=\sqrt{\frac{2B_g}{\rho}+\left(\frac{C_\lambda}{\rho}\right)^2}$, and $C_2:=\sqrt{\tfrac{2}{\rho}}\sum_{i=0}^\infty
	\sqrt{\alpha_i} +{ (L_{h}+M_{h}) \|\theta_0-\theta^*\| \over 1-q}$.
\end{theorem}
\begin{proof}
%For  $i=1,\hdots,k$, $F_0(x_i,u_i;\theta_{i})$ is finite at
Let $u_i:=\nabla_\lambda{\cal L}_\rho(x_i,\lambda_i;\theta_i)$ for all $i\geq 0$. Note that computing $\nabla_\lambda{\cal L}_\rho$ using \eqref{Aug_L_formu}, we get $u_i=h\left(x_i;\theta_i\right)+\Pi_{\R_+^m}\left(-\tfrac{\lambda_i}{\rho}-h\left(x_i;\theta_i\right)\right)$;
hence, it trivially follows that
  \begin{align}h_j(x_i,\theta_{i})\leq [u_i]_j, \hbox{ for } j=1,\hdots,m. \label{bd-u} \end{align}
 Under Assumption~\ref{Assump: Aug_1}(iv), we have that
$$ \left| h_j(x_i,\theta_i) - h_j(x_i,\theta^*) \right| \leq
L_{h}^j \|\theta_i - \theta^*\|.$$
%implying that $h_j(x_i,\theta_{i})\geq h_j(x_i,\theta^*)-L^j_{h}\|\theta_{i}-\theta^*\|$.
Combining this with \eqref{bd-u}, we obtain
 \begin{align}
 \label{bd-u2}
 h_j(x_i,\theta^*)\leq
 [u_i]_j+L^j_{h}\|\theta_{i}-\theta^*\|.\end{align}
 By summing \eqref{bd-u2} from $i=0$ to $i=k$, it follows that
   \begin{align}
   \label{bd-h}
   \sum_{i=0}^k h_j(x_i,\theta^*)\leq \sum_{i=0}^k [u_i]_j+\sum_{i=0}^k
   L^j_{h}\|\theta_{i}-\theta^*\| .\end{align}
 On the other hand, convexity of $h_j(x,\theta^*)$ in $x$ implies that
  $$h_j(\xbar_k,\theta^*)\leq {1\over k+1} \sum_{i=0}^k h_j(x_i,\theta^*);$$
 hence, for all $j=1,\dots,m$, we have from \eqref{bd-h},
\begin{align}
h_j(\xbar_k,\theta^*) \leq {1\over k+1}\left(\sum_{i=0}^k [u_i]_j+\sum_{i=0}^k L^j_{h}\|\theta_{i}-\theta^*\|\right). \label{bd-h2}
%\notag &\leq {1\over k+1}\sum_{i=0}^k (u_i)_j+{1\over k+1} \left(\max_{j=1,\hdots,m} L^j_{h,\theta}\right) \sum_{i=0}^k\|\theta_{i}-\theta^*\|\\
%&\leq {1\over k+1}\left(\sum_{i=0}^k [u_i]_j+L_{h}\sum_{i=0}^k\|\theta_{i}-\theta^*\|\right),
 \end{align}
Hence, $L_{h} \triangleq \max\{L^j_{h}:\ j=1,\hdots,m\}$, and \eqref{bd-h2} imply that
\begin{align}
& \quad d_{\R_-^m}\Big(h(\xbar_k,\theta^*)\Big)
\leq {1\over k+1}\left(\left\|\sum_{i=0}^k u_i\right\|+ L_{h}\sum_{i=0}^k \|\theta_{i}-\theta^*\|\right)\notag\\
&\leq \frac{1}{k+1}\left(\sum_{i=0}^k \|u_i\|+ L_{h}\sum_{i=0}^k \|\theta_{i}-\theta^*\|\right).\label{proof:bnd_prim_infea_2}
\end{align}
Recall that from Lemma~\ref{Lemma: Aug_Rock_Lemmas} (iii), for $i=0,\hdots,k,$
$$\Big\|\nabla_{\lambda} {\cal
	L}_\rho(x_i,\lambda_i;\theta_{i})-\nabla_{\lambda}
	g_\rho(\lambda_i;\theta_{i})\Big\|\leq \sqrt{\frac{2\alpha_i}{\rho}};$$
therefore, we obtain that $\|u_i\|=\|\nabla_{\lambda} {\cal
	L}_\rho(x_i,\lambda_i;\theta_{i})\|\leq \|\nabla_{\lambda}
	g_\rho(\lambda_i;\theta_{i})\|+\sqrt{2\alpha_i/\rho}$. In addition, since  $\|\nabla_{\lambda} g_\rho(\lambda_i;\theta_{i})\|\leq \|\nabla_{\lambda} g_\rho(\lambda_i;\theta^*)\|+M_{h}\|\theta_{i}-\theta^*\|$, we get the following bound:
$$\|u_i\|\leq \|\nabla_{\lambda}
g_\rho(\lambda_i;\theta^*)\|+\sqrt{2\alpha_i/\rho}+M_{h,\theta}\|\theta_{i}-\theta^*\|.$$
On the other hand, by Lemma~\ref{Lemma:Aug_grad_bnd}, we have
$$ \|\nabla_{\lambda} g_\rho(\lambda_i;\theta^*)\|\leq \sqrt{ \tfrac{2}{\rho} \left(f^*-g_\rho(\lambda_i;\theta^*)\right)}.$$
Combining this with the previous inequality leads to
\begin{align}
\|u_i\| & \leq \sqrt{\frac{2}{\rho}
	\Big(f^*-g_\rho(\lambda_i;\theta^*)\Big)}+\sqrt{\frac{2\alpha_i}{\rho}}+M_{h}\|\theta_{i}-\theta^*\|.\notag
\end{align}
By substituting this bound into~\eqref{proof:bnd_prim_infea_2}, we get that
\begin{align}
& \quad d_{\R_-^m}\Big(h(\xbar_k,\theta^*)\Big) \leq {1\over k+1}\sum_{i=0}^k \sqrt{ {2\over \rho} \Big(f^*-g_\rho(\lambda_i;\theta^*)\Big)}\notag \\
& \quad + {1\over k+1}\left(\sum_{i=0}^k \sqrt{\frac{2\alpha_i}{\rho}}+(L_{h}+M_{h})\sum_{i=0}^k \|\theta_{i}-\theta^*\|\right)\notag\\
&\leq \sqrt{ {2\over \rho} \left(f^*-{1\over k+1}\sum_{i=0}^kg_\rho(\lambda_i;\theta^*)\right)}\notag \\
&\quad+{1\over k+1}\left(\sum_{i=0}^k \sqrt{\frac{2\alpha_i}{\rho}}+(L_{h}+M_{h})\sum_{i=0}^k \|\theta_{i}-\theta^*\|\right),
\label{proof:bnd_prim_infea_5}
\end{align}
where the last inequality follows from concavity of square-root function $\sqrt{\cdot}$. The first term in
\eqref{proof:bnd_prim_infea_5} can be bounded using \eqref{proof: bnd_dual_sub_opt}, which states that
\begin{align}
&f^*-{1\over k+1}\sum_{i=0}^kg_\rho(\lambda_i;\theta^*)\notag\\&\leq
  {1 \over k+1}\left(B_g+g_\rho(\lambda_0;\theta^*)-g_\rho(\lambda_{k+1};\theta^*)\right).
\label{proof:bnd_prim_infea_3}
\end{align}
Note that $g_\rho(\lambda_0;\theta^*)- f^*\leq 0$, and using Lipschitz continuity of $\nabla g_\rho$, we have $f^*-g_\rho(\lambda_{k+1};\theta^*)\leq\tfrac{1}{2\rho}\|\lambda_{k+1}-\lambda^*\|^2\leq \tfrac{1}{2\rho} C_\lambda^2$.
The remaining terms in
\eqref{proof:bnd_prim_infea_5} can also be bounded:
\begin{eqnarray}
\lefteqn{{1\over k+1}\left(\sum_{i=0}^k \sqrt{\frac{2\alpha_i}{\rho}} +
(L_{h}+M_{h})\sum_{i=0}^k \|\theta_{i}-\theta^*\|\right)\leq}\notag\\
& & {1\over k+1}\left[\sum_{i=0}^\infty \sqrt{\frac{2\alpha_i}{\rho}} +{ (L_{h}+M_{h})
	\|\theta_0-\theta^*\| \over 1-q}\right].\label{proof:bnd_prim_infea_4}
\end{eqnarray}
The result follows by incorporating these bounds %~\eqref{proof:bnd_prim_infea_3} and~\eqref{proof:bnd_prim_infea_4}
into \eqref{proof:bnd_prim_infea_5}.
\end{proof}
We now proceed to derive
and upper bounds on
$f(\xbar_k,\theta^*)-f^*$. In contrast with standard
unconstrained convex optimization, $f(\xbar_k, \theta^*)$ could be less than
$f^*$, as a consequence of infeasibility of $\xbar_k$.

\begin{theorem}[{\bf bound on primal suboptimality}]
\label{Theorem: Aug_prim_low_bnd_cons_rho}
Let Assumption~\ref{Assump: Aug_1}--\ref{Assump:Aug_2} hold and let $\{x_k\}$ and $\{\lambda_k\}$ be the
sequences generated by Algorithm~\ref{Alg: Aug_const_rho}. In addition,
let $\xbar_k={1\over k+1}\sum_{i=0}^k x_k$. Then the following holds:
\begin{align}
f(\xbar_k;\theta^*)-f^* & \geq -{\rho \over 2} {\cal V}^2(k)-\|\lambda^*\|{\cal V}(k)\\ %(x^*;\theta^*)
f(\xbar_k;\theta^*)- f^* & \leq \frac{U}{k},
\end{align}
for any $\lambda^*\in\argmax\, g_0(\lambda,\theta^*)$, where ${\cal V}(k)$ is defined in Theorem~\ref{Theorem:Aug_prim_infeas_cons_rho},
$U:=\tfrac{\rho}{2}L^2_{h}\frac{\|\theta_0-\theta^*\|^2}{1-q_\ell^2}+\left(\bar{C}L_{h}+2L_{f}\right)\frac{\|\theta_0-\theta^*\|}{1-q_\ell}+\sum_{i=0}^\infty \alpha_i$.
\end{theorem}
\begin{proof}
\noindent {\bf Proof of the lower bound.} Since $\sup_\lambda g_\rho(\lambda;\theta^*)=\min_{x\in X} {\cal L}_\rho(x,\lambda^*;\theta^*)=f^*$, we have that for all $k\geq 0$,
\begin{align*}
f^*&\leq {\cal L}_\rho(\xbar_k,\lambda^*;\theta^*)\\&=f(\xbar_k;\theta^*)+{\rho\over
		2}d^2_{\mathbb{R}^m_-}\Big(h(\xbar_k;\theta^*)+{\lambda^*\over \rho}\Big)-{\|\lambda^*\|^2\over 2\rho}
\\&\leq f(\xbar_k;\theta^*)+{\rho\over
	2}\left(d_{\mathbb{R}^m_-}\left(h(\xbar_k;\theta^*)\right)+{\|\lambda^*\|\over \rho}\right)^2-{\|\lambda^*\|^2\over 2\rho},
\end{align*}
where the first equality is a consequence of \eqref{Aug_L_formu} while
the second inequality follows from Lemma~\ref{Lemma:d_triangle}. By expanding the second term above inequality, we obtain
\begin{align*}
%f(x^*;\theta^*)
f^*  %\leq f(\xbar_k;\theta^*)-{\|\lambda^*\|^2\over 2\rho}\\
%& +{\rho\over 2}\Bigg[d_{\mathbb{R}^m_+}\Big(-h(\xbar_k;\theta^*)\Big)^2+\Big({\|\lambda^*\|\over \rho}\Big)^2+2d_{\mathbb{R}^m_+}\Big(-h(\xbar_k;\theta^*)\Big){\|\lambda^*\|\over \rho}\Bigg]\\
&\leq f(\xbar_k;\theta^*)+\tfrac{\rho}{2}d^2_{\mathbb{R}^m_-}\left(h(\xbar_k;\theta^*)\right)+d_{\mathbb{R}^m_-}\left(h(\xbar_k;\theta^*)\right)\|\lambda^*\|\\
&\leq f(\xbar_k;\theta^*)+{\rho \over 2} {\cal V}^2(k)+\|\lambda^*\|{\cal V}(k),
\end{align*}
where the last inequality follows from
Theorem~\ref{Theorem:Aug_prim_infeas_cons_rho}.\\ %The result follows from a rearrangement of the terms.
\noindent {\bf Proof of the upper bound.} Let $x^*$ be an optimal solution to $\mathcal{C}(\theta^*)$. Step 1 of Algorithm~\ref{Alg: Aug_const_rho} implies that for all $i\geq 0$
\begin{align*}
 {\cal L}_\rho(x_i,\lambda_i;\theta_i)
 %&\leq  \inf_{x\in X} {\cal L}_\rho(x,\lambda_k;\theta_{k})+\alpha_k
 \leq  {\cal L}_\rho(x^*,\lambda_i;\theta_i)+\alpha_i.
 \end{align*}
 Hence, by the definition of ${\cal L}_\rho$, it follows that
\begin{eqnarray*}
\lefteqn{f(x_i;\theta_i)+{\rho\over 2}d^2_{\mathbb{R}^m_-}\left(h(x_i;\theta_i)+{\lambda_i\over \rho}\right)-{\|\lambda_i\|^2\over 2\rho}\leq}\\
& & f(x^*;\theta_i)+{\rho\over 2}d^2_{\mathbb{R}^m_-}\left(h(x^*;\theta_i)+{\lambda_i\over \rho}\right)-{\|\lambda_i\|^2\over 2\rho}+\alpha_i,
\end{eqnarray*}
which leads to
\begin{eqnarray}
\lefteqn{f(x_i;\theta_i)-f(x^*;\theta_i) \leq} \label{proof:primal_sub_opt_cons_rho1}\\
& & \tfrac{\rho}{2}d^2_{\mathbb{R}^m_-}\left(h(x^*;\theta_i)+{\lambda_i\over \rho}\right)-\tfrac{\rho}{2}d^2_{\mathbb{R}^m_-}\left(h(x_i;\theta_i)+{\lambda_i\over \rho}\right)+\alpha_i.\notag
\end{eqnarray}
Step 2 of Algorithm~\ref{Alg: Aug_const_rho} implies that
\begin{align}
d_{\mathbb{R}^m_-}\left(h(x_i;\theta_i)+{\lambda_i\over \rho}\right)={\|\lambda_{i+1}\|\over \rho}.
\label{proof:primal_sub_opt_cons_rho2}
\end{align}
In addition, by using Lemma~\ref{Lemma:d_triangle}, it follows that
\begin{align}
d_{\mathbb{R}^m_-}\left(h(x^*;\theta_i)+{\lambda_i\over \rho}\right)\leq d_{\mathbb{R}^m_-}\left(h(x^*;\theta_i)\right)+{\|\lambda_i\|\over \rho}.
\label{proof:primal_sub_opt_cons_rho3}
\end{align}
Substituting \eqref{proof:primal_sub_opt_cons_rho2}
and \eqref{proof:primal_sub_opt_cons_rho3} in \eqref{proof:primal_sub_opt_cons_rho1}, we obtain for all $i\geq 0$
\begin{align}
&f(x_i;\theta_i)-f(x^*;\theta_i)
  \leq {\rho \over 2}\left(d_{\mathbb{R}^m_-}\left(h(x^*;\theta_i)\right)+{\|\lambda_i\|\over \rho}\right)^2\notag\\&-{1 \over 2\rho}\|\lambda_{i+1}\|^2+\alpha_i  ={\rho \over 2}~d^2_{\mathbb{R}^m_-}\left(h(x^*;\theta_i)\right)+d_{\mathbb{R}^m_-}\left(h(x^*;\theta_i)\right)\|\lambda_i\|\notag\\
& \quad+{1\over 2\rho}\Big(\|\lambda_i\|^2-\|\lambda_{i+1}\|^2\Big)+\alpha_i.
\label{proof:primal_sub_opt_cons_rho5}
\end{align}
From Lipschitz continuity of $h_j$ in $\theta$ for $j=1,\hdots,m,$
\begin{align}\notag
h_j(x^*;\theta_i) & \leq h_j(x^*;\theta^*)+L_{h}\|\theta_i-\theta^*\|;\\
\implies d_{\mathbb{R}^m_-}\left(h(x^*;\theta_i)\right) & \leq
d_{\mathbb{R}^m_-}\left(h(x^*;\theta^*)\right)+L_{h}\|\theta_i-\theta^*\|.
\label{proof:primal_sub_opt_cons_rho4}
\end{align}
Since $h_j(x^*;\theta^*)\leq 0$ for $j=1,\hdots,m$, it follows that
$d_{\mathbb{R}^m_-}(h(x^*;\theta^*))=0$,
and inequality~\eqref{proof:primal_sub_opt_cons_rho4} becomes
\begin{align}
d_{\mathbb{R}^m_-}\left(h(x^*;\theta_i)\right)\leq L_{h}\|\theta_i-\theta^*\|.\notag
\end{align}
By substituting \eqref{proof:primal_sub_opt_cons_rho4} into~\eqref{proof:primal_sub_opt_cons_rho5}, we get for all $i\geq 0$
\begin{align}
f(x_i;\theta_i)-f(x^*;\theta_i)\leq\tfrac{\rho}{2}L^2_{h}\|\theta_i-\theta^*\|^2+\bar{C}L_{h}\|\theta_i-\theta^*\|\notag\\
 \quad+\tfrac{1}{2\rho}\left(\|\lambda_i\|^2-\|\lambda_{i+1}\|^2\right)+\alpha_i,\notag
\end{align}
where the last inequality follows from $\|\lambda_i-\lambda^*\| \leq C_{\lambda}$
(Theorem~\ref{Theorem: Aug_lambda_bnd_cont_rho}), i.e., $\|\lambda_i\|\leq \bar{C}:=C_\lambda+\|\lambda^*\|$ for all $i\geq 0$. Next, from the Lipschitz continuity of $f$ in $\theta$, it follows that
$$f(x_i;\theta_i)-f(x^*;\theta_i)\geq f(x_i;\theta^*)-f(x^*;\theta^*)-2L_{f}\|\theta_i-\theta^*\|.$$
Combining two above inequalities results in the following:
\begin{align*}
f(x_i;\theta^*)-f^*&\leq
 \tfrac{\rho}{2}L^2_{h}\|\theta_i-\theta^*\|^2+\left(\bar{C}L_{h}+2L_{f}\right)\|\theta_i-\theta^*\|\\
&  {1\over 2\rho}\Big(\|\lambda_i\|^2-\|\lambda_{i+1}\|^2\Big)+\alpha_i.
\end{align*}
Summing the above inequality for $i=0$ to $k$, we get
\begin{align*}
&\sum_{i=0}^k\Big(f(x_i;\theta^*)-f^* \Big)\leq\tfrac{\rho}{2}L^2_{h}\sum_{i=0}^k \|\theta_{i}-\theta^*\|^2+\sum_{i=1}^k \alpha_i\\&+\left(\bar{C}L_{h}+2L_{f}\right)\sum_{i=0}^k \|\theta_{i}-\theta^*\|
+\tfrac{1}{2\rho}\sum_{i=0}^k\left(\|\lambda_i\|^2-\|\lambda_{i+1}\|^2\right)\\
%&\leq\tfrac{\rho}{2}L^2_{h}\sum_{i=0}^\infty \|\theta_{i}-\theta^*\|^2+\Big[L_{h}(C_\lambda+\|\lambda^*\|)+2L_{f}\Big]\sum_{i=0}^\infty \|\theta_{i}-\theta^*\|\\
%&\quad+\tfrac{1}{2\rho}\|\lambda_0\|^2+\sum_{i=0}^\infty \alpha_i\\
&\leq\tfrac{\rho}{2}L^2_{h}\frac{\|\theta_0-\theta^*\|^2}{1-q^2}+\left(\bar{C}L_{h}+2L_{f}\right)\frac{\|\theta_0-\theta^*\|}{1-q}%+\tfrac{1}{2\rho}\|\lambda_0\|^2
+\sum_{i=0}^\infty \alpha_i.
\end{align*}
Since $f(x;\theta)$ is convex in $x$, dividing both sides of the above inequality by $k$ gives the desired result.
%\begin{align*}
%&\quad f(\xbar_k;\theta^*)-f(x^*;\theta^*) \\
%	& \leq{1 \over k+1}\Big({\rho \over 2}L^2_{h,\theta}
%			{\|\theta_0-\theta^*\|^2\over
%			1-q^2_\ell}+(L_{h,\theta}C_\lambda\\
%	&+2L_{f,\theta}){\|\theta_0-\theta^*\|\over 1-q_\ell}+{1 \over 2\rho}\|\lambda_0\|^2+\sum_{i=0}^\infty \alpha_i\Big).
%\end{align*}
\end{proof}

\section{Conclusion}\label{Sec:4}
In this paper, we consider the setting of an optimization problem
	complicated by misspecification both in the function and in the
		constraints. The %misspecified
parameter misspecification may be resolved through
		a learning problem. Suppose we have %that arises from
		%having
access to a learning
data set, collected a priori. One avenue for contending with
such a problem is through an inherently sequential approach
that solves the learning problem and utilizes this solution
in subsequently solving the computational problem. Unfortunately,
unless accurate solutions of the learning problem
are available in a finite number of iterations,
sequential approaches can only provide approximate solutions. Instead, we focus on a simultaneous approach that combines learning and computation by adopting inexact augmented Lagrangian (AL) scheme with constant penalty parameter. In this regard, we made the following contributions:
(i) Derivation of the convergence rate for dual optimality, primal infeasiblity and primal suboptimality;
(ii) Quantification of the effect of learning on the rate.

Our future work lies in deriving the overall iteration complexity
analysis, which incorporates the number of iterations required to
solve the subproblems arising in the AL scheme and quantify the
	resulting impact of learning.
\bibliographystyle{IEEEtran}
\bibliography{ref2}

\def\cprime{$'$}
\begin{thebibliography}{10}
\providecommand{\url}[1]{#1}
\csname url@samestyle\endcsname
\providecommand{\newblock}{\relax}
\providecommand{\bibinfo}[2]{#2}
\providecommand{\BIBentrySTDinterwordspacing}{\spaceskip=0pt\relax}
\providecommand{\BIBentryALTinterwordstretchfactor}{4}
\providecommand{\BIBentryALTinterwordspacing}{\spaceskip=\fontdimen2\font plus
\BIBentryALTinterwordstretchfactor\fontdimen3\font minus
  \fontdimen4\font\relax}
\providecommand{\BIBforeignlanguage}[2]{{%
\expandafter\ifx\csname l@#1\endcsname\relax
\typeout{** WARNING: IEEEtran.bst: No hyphenation pattern has been}%
\typeout{** loaded for the language `#1'. Using the pattern for}%
\typeout{** the default language instead.}%
\else
\language=\csname l@#1\endcsname
\fi
#2}}
\providecommand{\BIBdecl}{\relax}
\BIBdecl

\bibitem{bental09robust}
A.~Ben-Tal, L.~El~Ghaoui, and A.~Nemirovski, \emph{Robust optimization}, ser.
  Princeton Series in Applied Mathematics.\hskip 1em plus 0.5em minus
  0.4em\relax Princeton University Press, Princeton, NJ, 2009.

\bibitem{shapiro09lectures}
A.~Shapiro, D.~Dentcheva, and A.~Ruszczy{\'n}ski, \emph{Lectures on stochastic
  programming}, ser. MPS/SIAM Series on Optimization.\hskip 1em plus 0.5em
  minus 0.4em\relax Philadelphia, PA: Society for Industrial and Applied
  Mathematics (SIAM), 2009, vol.~9, modeling and theory.

\bibitem{jiang13solution}
H.~Jiang and U.~V. Shanbhag, ``On the solution of stochastic optimization
  problems in imperfect information regimes,'' in \emph{Winter Simulations
  Conference: Simulation Making Decisions in a Complex World, {WSC}}, 2013, pp.
  821--832.

\bibitem{jiang16nash}
H.~Jiang, U.~V. Shanbhag, and S.~P. Meyn, ``Distributed computation of
  equilibria in misspecified convex stochastic {N}ash games,''
  \emph{Conditionally accepted in IEEE Transactions on Automatic Control},
  2016.

\bibitem{ahmadi14data}
H.~Ahmadi and U.~V. Shanbhag, ``Data-driven first-order methods for
  misspecified convex optimization problems: Global convergence and rate
  estimates,'' in \emph{53rd {IEEE} Conference on Decision and Control, {CDC}},
  2014, pp. 4228--4233.

\bibitem{ahmadi15misspecified}
------, ``On the resolution of misspecified convex optimization and monotone
  variational inequality problems,'' \emph{http://arxiv.org/abs/1408.5532}.

\bibitem{hestenes1969multiplier}
M.~R. Hestenes, ``Multiplier and gradient methods,'' \emph{Journal of
  optimization theory and applications}, vol.~4, no.~5, pp. 303--320, 1969.

\bibitem{Powell69_1J}
M.~Powell, \emph{Optimization}.\hskip 1em plus 0.5em minus 0.4em\relax Academic
  Press, London/New York, 1969, ch. A method for nonlinear constraints in
  minimization problems, p. 283–298.

\bibitem{Rockafellar73_1J}
R.~Rockafellar, ``The multiplier method of {H}estenes and {P}owell applied to
  convex programming,'' \emph{Journal of Optimization Theory and Applications},
  vol.~12, no.~6, pp. 555--562, 1973.

\bibitem{rockafellar1973dual}
R.~T. Rockafellar, ``A dual approach to solving nonlinear programming problems
  by unconstrained optimization,'' \emph{Mathematical Programming}, vol.~5,
  no.~1, pp. 354--373, 1973.

\bibitem{aybat2013augmented}
N.~S. Aybat and G.~Iyengar, ``An augmented {L}agrangian method for conic convex
  programming,'' \emph{arXiv preprint arXiv:1302.6322}, 2013.

\bibitem{lan15_1J}
G.~Lan and R.~Monteiro, ``Iteration-complexity of first-order augmented
  {L}agrangian methods for convex programming,'' \emph{Mathematical
  Programming}, pp. 1--37, 2015.

\bibitem{1506.05320}
I.~Necoara, A.~Patrascu, and F.~Glineur, ``Complexity certifications of first
  order inexact lagrangian and penalty methods for conic convex programming,''
  \emph{http://arxiv.org/abs/1506.05320}.

\bibitem{1510.00490}
H.~Ahmadi, N.~S. Aybat, and U.~V. Shanbhag, ``On the analysis of inexact
  augmented lagrangian schemes for misspecified conic convex programs,''
  \emph{http://arxiv.org/abs/1608.01879}.

\bibitem{lee2010quadratic}
G.~Lee, N.~Tam, and N.~Yen, \emph{Quadratic Programming and Affine Variational
  Inequalities: A Qualitative Study}, ser. Nonconvex Optimization and Its
  Applications.\hskip 1em plus 0.5em minus 0.4em\relax Springer US, 2010.

\bibitem{facchinei02finite}
F.~Facchinei and J.-S. Pang, \emph{Finite-dimensional variational inequalities
  and complementarity problems. {V}ol. {I}}, ser. Springer Series in Operations
  Research.\hskip 1em plus 0.5em minus 0.4em\relax New York: Springer-Verlag,
  2003.

\bibitem{hiriart2001convex}
J.-B. Hiriart-Urruty and C.~Lemar{\'e}chal, \emph{Convex Analysis and
  Minimization Algorithms II: Advanced Theory and Bundle Methods}.\hskip 1em
  plus 0.5em minus 0.4em\relax Springer-Verlag, New York, 2001.

\bibitem{nesterov2013introductory}
Y.~Nesterov, \emph{Introductory lectures on convex optimization: A basic
  course}.\hskip 1em plus 0.5em minus 0.4em\relax Springer Science \& Business
  Media, 2013, vol.~87.

\end{thebibliography}
%%%%%%%%%%%%%%%%%%%%%%%%%%%%%%%%%%%%%%%%%%%%%%%%%%%%%%%%%%%%%%%%%%%%%%%

%%%%%%%%%%%%%%%%%%%%%%%%%%%%%%%%%%%%%%%%%%%%%%%%%%%%%%%%%%%%%%%%%%%%%%%
\end{document}